\documentclass{article}
\usepackage{fancybox,amsmath,amssymb,ascmac,amsthm}
\usepackage{geometry}
\usepackage{enumerate}
\usepackage[dvips]{graphicx,color}

\title{A new recurrence formula for generic exceptional orthogonal polynomials}
\author{Hiroshi Miki\thanks{Department of Electronics, Faculty of Science and Engineering, Doshisha University,
Kyotanabe City, Kyoto 610 0394, Japan\quad E-mail:hmiki@mail.doshisha.ac.jp}
\and Satoshi Tsujimoto\thanks{Department of Applied Mathematics and Physics, Graduate School of Informatics,
Kyoto University, Sakyo-Ku, Kyoto 606 8501, Japan\quad E-mail:tujimoto@i.kyoto-u.ac.jp}}
\date{}
\newtheorem{thm}{Theorem}[section]

\newtheorem{lmm}[thm]{Lemma}

\theoremstyle{definition}

\theoremstyle{remark}
\newtheorem*{rem}{Remark}

\begin{document}

\maketitle

\newcommand{\setG}{\protect{\mathbb G}}
\newcommand{\setH}{\protect{\mathbb H}}
\newcommand{\setI}{\protect{\mathbb I}}
\newcommand{\setL}{\protect{\mathbb L}}
\newcommand{\setP}{\protect{\mathbb P}}
\newcommand{\setC}{\protect{\mathbb C}}
\newcommand{\setD}{\protect{\mathbb D}}
\newcommand{\setZ}{\protect{\mathbb Z}}
\newcommand{\setN}{\protect{\mathbb N}}
\newcommand{\setR}{\protect{\mathbb R}}
\newcommand{\setS}{\protect{\mathbb S}}
\newcommand{\LL}{{\cal L}}
\newcommand{\DD}{{\cal D}}
\newcommand{\FF}{{\cal F}}
\newcommand{\TT}{{\cal T}}
\newcommand{\BB}{{\cal B}}
\newcommand{\II}{{\cal I}}
\newcommand{\dx}{\partial}
\newcommand{\dxdx}{\partial^2}
\newcommand{\Res}[1]{\underset{#1}{\mbox{{\rm Res}}}}
\newcommand{\gamZero}{\protect{{{\rm \varnothing}}}}
\newcommand{\gamThree}{\protect{{{\rm III}}}}
\newcommand{\gamOne}{\protect{{{\rm I}}}}
\newcommand{\gamTwo}{\protect{{{\rm II}}}}

\begin{abstract}
A new recurrence relation for exceptional orthogonal polynomials is proposed, which holds for type 1, 2 and 3. As concrete examples, the recurrence relations are given for $X_j$-Hermite, Laguerre and Jacobi polynomials in $j=1,2$ case.
\end{abstract}
\vskip 3mm
Keywords: exceptional orthogonal polynomials, Darboux transformation, recurrence relations
\\[3mm]
Mathematics Subject Classification: 33C45, 33C47, 42C05

\section{Introduction}

The exceptional orthogonal polynomials (XOPs) were introduced by G\'omez-Ullate, Kamran and Milson
as a generalization of the classical orthogonal polynomials (COPs), where the polynomials of the first several degree are missed\cite{gomez1}. After their introduction, the XOPs have been widely accepted and have been developed by many researchers.
In particular, Sasaki and Odake have figured out the direct relationship between XOPs and COPs, and then derived many exceptional extensions of COPs belonging to the Askey scheme\cite{os16}.  They have also introduced the multi-indexed orthogonal polynomials by using the repeated Darboux transformations of the XOPs\cite{OS2011multi}.

In this article we will discuss the recurrence relations, one of the fundamental properties, of the XOPs.
The XOPs satisfy the second-order differential (or difference) equation by construction, although they do not hold three-term recurrence relations which ordinary OPs are supposed to. 
Concerning the XOPs derived from the one-time Darboux transformation of the COPs,  to the best of our knowledge, 
the shortest recurrence relations are
$4j+1$-term recurrence relations:
\begin{align*}
 p_j^2(x) Q_n(x)=\sum_{k=n-2j}^{n+2j} \alpha_{n,k} Q_{k}(x),
\end{align*}
where all coefficients $c_{k}^{(n)}$ do not depend on $x$.
Remark that, if the coefficients of recurrence relations allow to depend on $x$,
then there exist 5-term recurrence relations.

The purpose of this article is to present the new recurrence relations of 
the form:
\begin{align*}
 \left(\int p_j(x)dx\right)  Q_n(x) 
 =\sum_{k=n-j-1}^{n+j+1} \beta_{n,k} Q_{k}(x),
\end{align*}
which form $2j+3$-term and hold for any type of XOPs.

This article is organized as follows.
In \S 2, we review how to construct XOPs by using Darboux transformations \cite{adler,crum,darboux}. In \S 3, we show the new recurrence formula for XOPs by observing the structure of the Darboux transformation.
In \S 4, we give the concrete examples of the obtained recurrence relations 
especially for classical XOPs (Hermite, Laguerre and Jacobi).
Finally we give brief summary of this article.

\section{The exceptional orthogonal polynomials}
The exceptional orthogonal polynomials can be characterized by
Darboux transformations among the quasi-polynomial eigenfunctions of the classical B\"ochner type operator\cite{gomez2,SasakiTZ2010}. 
Those are the natural extensions of the COPs characterized by pure 
polynomial eigenfunctions of the classical B\"ochner type operator.
 In this section, we briefly review how to construct the exceptional orthogonal polynomials\cite{gomez2,os16}.

\subsection{Classical orthogonal polynomials}
We first explain COPs shortly before XOPs (see e.g. \cite{Nikiforov} for further details). COPs $\{ P_n(x)\}_{n=0}^{\infty }$ are usually defined as polynomial solutions to the following Sturm-Liouville problem:
\begin{eqnarray}
\label{polyeigen}
(A(x)\partial^2+B(x)\partial +C(x)) P_n
  = \lambda_n P_n,\
\quad \deg\, P_n = n,
\end{eqnarray}
or equivalently
\begin{eqnarray}
 \label{polyeigen_st}
 w(x)^{-1} \left(A(x) w(x) P_n'(x)\right)' = \lambda_n P_n(x),
\end{eqnarray}
where $\partial \equiv \frac{d}{dx}$ and  
\begin{eqnarray}\label{c-weight}
 (A(x)w(x))'=B(x)w(x).
\end{eqnarray}
In order that \eqref{polyeigen} admits the polynomial sequence solution, $A(x),B(x),C(x)$ are supposed to be 
\begin{eqnarray}\label{deg}
\deg (A(x))\le 2,\quad \deg(B(x))=1,\quad \deg(C(x))=0
\end{eqnarray}
and the polynomial solutions are found to be orthogonal polynomials whose orthogonality is described by
\begin{eqnarray}
\int_D P_m(x)P_n(x)w(x)dx=h_n\delta_{mn}\quad (h_n\ne 0),
\end{eqnarray} 
where 
$D \subset \mathbb{R}\cup \{ \pm \infty \}$
 is the corresponding interval s.t. $A(x)w(x)|_{x\in \partial {D}}=0$.
 One of the specific feature of the COPs is that the derivatives of the COPs $\{ P_n'(x)\}_{n=1}^{\infty }$are again COPs whose orthogonality is given by
 \begin{eqnarray}\label{derivative-ops}
 \int_D P_m'(x)P_n'(x)A(x)w(x)dx=\tilde{h}_n\delta_{mn} \quad (\tilde{h}_n \ne 0).
 \end{eqnarray} 
As is well-known, B\"{o}chner classified the all polynomials for \eqref{polyeigen} and showed up to Affine transformation the polynomial eigenfunctions belong to the classical
orthogonal polynomials of Jacobi, Laguerre, Hermite or the Bessel polynomials\cite{bochner}. Throughout this paper, we call the operator $A(x)\partial^2 +B(x)\partial +C(x)$ with \eqref{deg} as B\"{o}chner type operator.

\subsection{Quasi-polynomial eigenfunctions}
As is mentioned before, XOPs can be characterized by Darboux transformations among the quasi-polynomial eigenfunctions of the classical B\"ochner type operator. Hence we shall examine all the possible quasi-polynomial eigenfunctions\cite{Tsuji2013}.

Let $\setL$  be the set of B\"ochner type operators defined by
\begin{eqnarray}
\setL = \left\{\alpha_2 \dxdx+\alpha_1 \dx +\alpha_0\mid
\alpha_1,\alpha_2 \in \setC[x],
0 \le \deg(\alpha_2) \le 2,  \deg (\alpha_1) =1, \alpha_0 \in \setC
\right\}.
\end{eqnarray} 
We consider a special class of eigenfunctions $\phi(x)$ of $\LL \in \setL$ which can be separated into a gauge part $\xi(x)$ and a polynomial part $p(x)$, that is 
\begin{eqnarray*}
 \phi(x)= \xi(x) p(x).
\end{eqnarray*}
In the case of when $\xi'(x)/\xi(x)$ comes to be a rational function, we call $\phi(x)= \xi(x) p(x)$
a {\it quasi-polynomial} eigenfunction.

The gauge part of quasi-polynomial eigenfunctions can be determined
by
\begin{eqnarray}
\label{pr:xi:eta}
{\xi'(x;A,B,\gamma)}/{\xi(x;A,B,\gamma)} = \eta(x,A,B,\gamma),
\end{eqnarray}
where 
\begin{eqnarray}
\label{def:eta}
 \eta(x;A,B,\gamma)=\frac{1}{2\pi i}\oint_{\gamma} \frac{B(z) - A'(z)}{A(z)} \frac{dz}{z-x},
\end{eqnarray}
and $\gamma$ is a positively oriented closed curve in ${\mathbb C}
 \backslash \{\mbox{zeros of $A(z),x$}\}$ which does not enclose the point
 $x\in{\mathbb C}$.
Then it holds that
\begin{eqnarray}
\label{pr:cond:L}
  {\xi^{-1}}\circ \LL\circ  {\xi}\,\in \setL,
\end{eqnarray} 
where $\xi(x;A,B,\gamma)$ is formally given by
 \begin{eqnarray*}
&&
  \xi_{\gamma}(x)
=\xi(x;A,B,\gamma)
=\exp\left({\int \eta(x;A,B,\gamma)dx}\right).
\end{eqnarray*}
Hereafter we employ 
the notation $\xi_{\gamma}(x)=\xi(x;A,B,\gamma)$, $\eta_{\gamma}(x)=\eta(x;A,B,\gamma)$  for simplicity.
For later usage, let us introduce the function $A_{\gamma}(x)$ given by
\begin{eqnarray}
\label{def:A_gamma}
 A_{\gamma}'(x)/A_{\gamma}(x) = -\eta(x;A,0,\gamma).
\end{eqnarray}

After fixing the gauge part of quasi-polynomial eigenfunctions, the polynomial part $p(x)$ of any degree $n$ in $x$
can be determined 
by  solving the eigenvalue problem
\begin{eqnarray*}
  (\xi^{-1}\circ \LL\circ\xi)[p](x) = \lambda \,p(x), 
\end{eqnarray*}
up to the scaling constant, if $n A''/2+ B' + 2(A\xi'/\xi)' \ne 0$ for $n \in \setZ_{\ge 0}$.
Thus we may denote the $n$th degree monic polynomial eigenfunction of 
$\xi_{\gamma}^{-1}\circ\LL\circ\xi_{\gamma} = \LL + 2A\eta_{\gamma}\partial \in \setL $ by
$p_{\gamma,n}(x)$, that is
\begin{eqnarray}\label{bochner:polyseed}
 (\LL + 2A\eta_{\gamma}\partial  )[p_{\gamma,n}(x)]=\lambda_{\gamma,n}p_{\gamma,n}(x).
\end{eqnarray}

Let $\setG$ be the set of representative closed curves which determine
the gauge part $\xi(x)$ of quasi-polynomial eigenfunctions of $\LL$.
Then we introduce the set of quasi-polynomial eigenfunctions of $\LL$ 
by
\begin{eqnarray}
\label{def:Phi}
 \Phi=\{ \phi_{\gamma,n}(x)=\xi_{\gamma}(x) p_{\gamma,n}(x) \mid (\gamma,n) \in \setG \times \setZ_{\ge 0}\}.
\end{eqnarray}
\begin{rem}
In the case when $\gamma=\gamZero$, then $\xi_{\gamZero}$ becomes a constant and therefore $\phi_{\gamZero,n}(x)$ belongs to COPs.
\end{rem}

Since $\deg(A)\le 2$ and $\deg(B)=1$, 
we can find all gauge factors to be classified into at most four kinds depending on the choice of contour $\gamma \in \setG$. Let ${\bold Q}$ be the set of all poles of the integrand
$(B(z)-A'(z))A(z)^{-1}(z-x)^{-1}$ in ${\mathbb C}\cup \{\infty\}$. 
The number of elements of ${\bold Q}$, denoted by $\sharp {\bold Q}$, is
taken from 0 to 3, since the degree of $A$ is less than or equal to 2.
One can find that the integrand is identically zero if $\sharp {\bold
Q}=0$ and also ${\bold Q}=\{x\}$ if $\sharp {\bold Q} = 1$.

For any $\LL \in \setL$  we have two types of curves, $\gamZero$ and $\gamThree$ as follows:
\begin{itemize}
 \item $\gamZero\in \setG$ does not enclose any points in ${\bold Q}$,
 \item $\gamThree\in \setG$ encloses all points but $x$ in ${\bold Q}$.
\end{itemize}
Note that, if $\sharp {\bold Q}$ is equal to  0 or 1, then 
the curves ${{\gamZero}}$ and ${{\gamThree}} $ can be treated as being the same curve.
Moreover, if the set ${\bold Q}\backslash\{x\}$ contains two
elements, say $a_1$ and $a_2$, then
two more additional representative closed curves
$\gamOne$ and $\gamTwo$ can be 
introduced as
\begin{itemize}
 \item $\gamOne\in \setG$ encloses $a_1$, but not $a_2$ or $x$,
 \item $\gamTwo\in \setG$  encloses $a_2$, but not $a_1$ or $x$.
\end{itemize}
We obtain the type 1 and type 2 XOPs 
by taking the quasi-polynomials derived from the curves $\gamOne$ and $\gamTwo$ 
as a seed function of the Darboux transformation discussed in the next subsection. The type 3 XOPs is derived from the quasi-polynomials whose gauge part is given by $\xi_{\gamThree}(x)$\cite{DubovEleonskiiKulagin}.
 
\begin{figure}[htbp]
\vspace*{-1mm}
 \begin{minipage}{0.5\hsize}
  \begin{center}
   \includegraphics[scale=0.8]{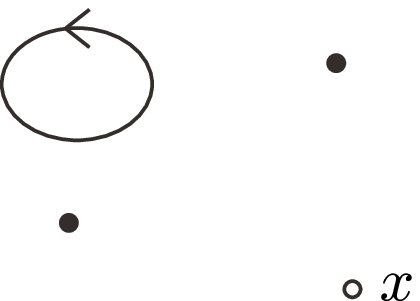}
  \end{center}
\vspace*{-5mm}
  \caption{Case $\sharp {\bold Q}=3$, curve $\gamZero$}
  \label{fig:one}
 \end{minipage}
 \begin{minipage}{0.5\hsize}
  \begin{center}
   \includegraphics[scale=0.8]{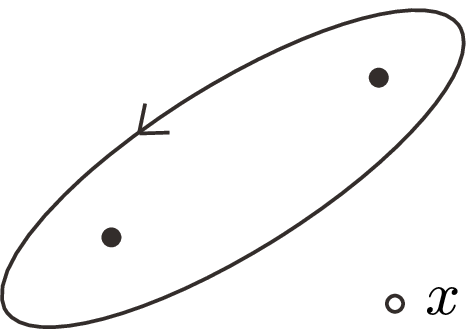}\\
  \end{center}
\vspace*{-5mm}
  \caption{Case $\sharp {\bold Q}=3$, curve $\gamThree$}
  \label{fig:two}
 \end{minipage}
\\[6mm]
 \begin{minipage}{0.5\hsize}
  \begin{center}
   \includegraphics[scale=0.8]{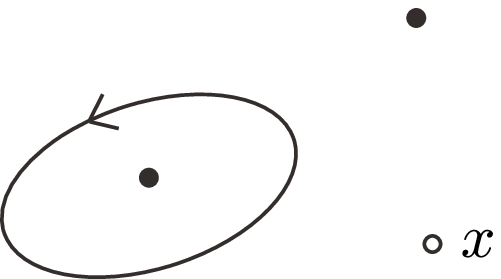}
  \end{center}
\vspace*{-5mm}
  \caption{Case $\sharp {\bold Q}=3$, curve $\gamOne$}
  \label{fig:three}
 \end{minipage}
 \begin{minipage}{0.5\hsize}
  \begin{center}
   \includegraphics[scale=0.8]{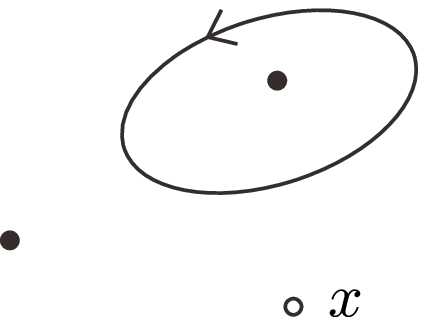}\\
  \end{center}
\vspace*{-5mm}
  \caption{Case $\sharp {\bold Q}=3$, curve $\gamTwo$}
  \label{fig:four}
 \end{minipage}
\end{figure}

\begin{lmm}
\label{lemma1}
Let $\Res{\zeta}=\Res{z=\zeta}\left[(B(z) - A'(z))/(A(z)(z-x))\right]$
be a residue at the point $\zeta\in {\mathbb C}\cup \{\infty\}$. 
For each $\LL\in\setL$, 
corresponding to the choice of the closed curve $\gamma \in \setG$, 
we have the following types of functions $\eta_{\gamma}(x)$:
\begin{itemize}
 \item For any pair of nonzero polynomial $A$ of degree at most 2, and
       linear polynomial $B$,\\
{\rm (i)}
     $\eta_{\gamZero}=0$, 
\qquad  {\rm (ii)}     $\eta_{\gamThree}(x) = -\Res{x}=(A'(x) - B(x))/A(x)$.
\end{itemize}
Only the following two cases admit additional two types of $\eta$:
\begin{itemize}
  \item 
For any pair of polynomials $A, B$ such that 
$A=a_0(x-a_1)(x-a_2)\ne 0$ with $a_1 \ne a_2$, $\deg B=1$, and 
$B-A'$ has no common root with $A$,
\\{\rm (iii)}
$\eta_{\gamOne}= \Res{a_1}$,
\qquad {\rm (iv)}
$\eta_{\gamTwo}= \Res{a_2}$.
 \item 
For any pair of polynomials $A, B$ such that 
$A=a_0(x-a_1)\ne 0$, $\deg B=1$, and 
$B-A'$ has no common root with $A$,
then
\\{\rm (iii)'}
 $\eta_{\gamOne}= \Res{a_1}$,
\qquad
{\rm (iv)'}
$\eta_{\gamTwo}= \Res{\infty}$.
\end{itemize}

\end{lmm}

\begin{rem}
 The cases (iii) and (iv) occur in the B\"ochner type operator of Jacobi polynomials, and also the cases (iii)' and (iv)' in the one of Laguerre polynomials.
\end{rem}

For a given closed curve $\gamma\in \setG$, there exists $\gamma^*\in \setG$  such that
$$\xi_{\gamma^*}(x)=\kappa (w(x)\xi_{\gamma}(x))^{-1}$$ with some nonzero constant $\kappa$.
By virtue of the arbitrary scaling factor of $\xi_{\gamma}$, the nonzero constant $\kappa$ can be set to $\kappa=1$. 
In a similarly manner for $A_{\gamma}$, 
one can obtain
$A_{\gamma^*}(x)=A(x)/A_{\gamma}(x)$.
Here and hereafter we assume that
\begin{eqnarray}
\label{rel:xi_A}
\xi_{\gamZero}(x)=A_{\gamZero}(x)=1, \quad
 w(x)= 1/(\xi_{\gamma}(x)\xi_{\gamma^*}(x)), \quad A(x)=A_{\gamma}(x)A_{\gamma^*}(x)
\end{eqnarray}
for any $\gamma \in \setG $.
One can find that 
$\gamThree$  is the dual of $\gamZero$, and,  if exists, $\gamTwo$ is the one
of $\gamOne$.
We may denote as follows
\begin{eqnarray*}
\gamZero=\gamThree^*,  \gamZero^*=\gamThree \mbox{\quad and \quad} 
\gamOne=\gamTwo^*, \gamOne^*=\gamTwo. 
\end{eqnarray*}
Analogously, we say that
$\xi_{\gamma^*}$ is the dual of $\xi_{\gamma}$.

\subsection{Darboux transformations}
Now we have a  B\"ochner type operator $\LL=A(x)\partial^2 + B(x)\partial  = w(x)^{-1} \partial\,  w(x)  A(x) \, \partial $  and its corresponding set $\Phi$ of polynomial and quasi-polynomial eigenfunctions  such that
\begin{eqnarray}\label{bochner:seed}
 \LL[ \phi_{\gamma,n}(x)] = \lambda_{\gamma,n} \phi_{\gamma,n}\quad \mbox{for all } (\gamma,n) \in \setG \times \setZ_{\ge 0}.
\end{eqnarray}
Here and hereafter we set the constant term $C=0$  on the ground that any B\"ochner type operator can be rewritten into the form (\ref{bochner:seed})
by replacing $\lambda_{\gamma,n}-C \to \lambda_{\gamma,n} $. 
Moreover we assume that spectral parameters $\lambda_{\gamma,n}$ are all mutually distinct for simplicity.

Let us take $\phi_{\rho,j}(x) \in \Phi$  to be the seed function of the Darboux transformation, where we call $(\rho,j)$ the {\it Darboux parameter}.
By using this seed function $\phi_{\rho,j}$, the operator $\LL$ can be factorized into 
\begin{eqnarray*}
  \LL = \BB \circ \FF + \lambda_{\rho,j}
\end{eqnarray*}
where 
\begin{eqnarray*}
&& \BB=A(x)\left(\partial +A'(x)/A(x)- w'(x)/w(x) - \phi_{\rho,j}'(x)/\phi_{\rho,j}(x)\right)\pi(x)^{-1},
\\
&& \FF=\pi(x)\left(\partial - {\phi_{\rho,j}'(x)}/{\phi_{\rho,j}(x)}\right) .
\end{eqnarray*}
Then the Darboux transformed operator and eigenfunctions  are respectively given by
\begin{eqnarray*}
\widehat \LL=\FF \circ\BB +\lambda_{\rho,j}
\end{eqnarray*}
and
\begin{eqnarray*}
\widehat \phi_{\gamma,n}(x)
={\cal T}_{\rho,j}[\phi_{\gamma,n}(x)]
=\left\{\begin{array}{lc}
	\FF[\phi_{\gamma,n}(x)]
=
\dfrac{A_{\rho}(x)}{\xi_{\rho}(x)}
{\rm Wr}[\phi_{\rho,j},\phi_{\gamma,n}](x)
&  \mbox{ if } (\gamma,n) \ne (\rho,j), \\[4mm]
A_{\rho}(x)/(A(x)w(x)\xi_{\rho}(x)) = 
\xi_{\rho^*}(x)/A_{\rho^*}(x)  
& \mbox{ if } (\gamma,n) = (\rho,j),
	      \end{array}
\right.
\end{eqnarray*}
where 
${\rm Wr}[f_1,f_2](x)=f_1(x)f_2'(x)-f_1'(x)f_2(x)$ and   we put the normalization factor $\pi(x)=A_{\rho}(x) \phi_{\rho,j}(x)/ \xi_{\rho}(x)$. 

It is easy to see that 
this transformation can be considered as an isospectral transformation of $\LL$ and $\phi_{\gamma,n}\in\Phi$: 
\begin{eqnarray*}
 \widehat \LL[\widehat \phi_{\gamma,n}(x)] = \lambda_{\gamma,n} \widehat \phi_{\gamma,n}(x)\quad \text{for all $(\gamma,n) \in \setG \times \setZ_{\ge 0}$}.
\end{eqnarray*}

Moreover the Darboux transformation ${\cal T}$ associated with XOPs preserves the quasi-polynomiality of eigenfunctions, that is,
\begin{eqnarray*}
 {\cal T} : \Phi \to \widehat \Phi
=\left\{ \widehat \phi_{\gamma,n}={\cal T}_{\rho,j}[\phi_{\gamma,n}] 
=\widehat \xi_{\gamma,n} \widehat p_{\gamma,n}
\mid (\gamma,n) \in \setG \times \setZ_{\ge 0}\right\},
\end{eqnarray*}
where $\widehat p_{\gamma,n} \in \setC[x]$ and 
\begin{eqnarray*}
\widehat \xi_{\gamma,n}(x)
=
\left\{\begin{array}{ll}
 \xi(x;A,B+{\rm Wr}[A_{\rho},A_{\rho^*}],\gamma)
& (\gamma,n) \ne (\rho,j)\\[1mm]
 \xi(x;A,B+{\rm Wr}[A_{\rho},A_{\rho^*}],\gamma^*)
&  (\gamma,n) = (\rho,j)
       \end{array}
\right..
\end{eqnarray*}

By construction the set of all polynomial eigenfunctions transformed with the Darboux parameter $(\rho,j)$ can be presented by
\begin{align*}
\begin{array}{ll}
  \left\{\widehat \phi_{\gamma,n} \mid (\gamma,n) \in \{\gamZero\} \times \setZ_{\ge 0}\backslash\{(\gamZero,j)\} \right\} \quad
  &  \text{if $\rho=\gamZero$},\\
  \left\{\widehat \phi_{\gamma,n} \mid (\gamma,n) \in \{\gamZero\} \times \setZ_{\ge 0} \right\} \quad
  &  \text{if $\rho=\gamOne$ or  $\gamTwo$},\\
  \left\{\widehat \phi_{\gamma,n} \mid (\gamma,n) \in \{\gamZero\} \times \setZ_{\ge 0} \bigcup \{(\gamThree,j)\} \right\} \quad
  &  \text{if $\rho=\gamThree$}.
\end{array}\end{align*}

\section{Recurrence relation for XOPs}

In the previous section, we have seen that so-called $X_j$-OPs of any type can be obtained via one Darboux transformation.
In the followings, we denote type 1, 2, 3 $X_j$-OPs by $\widehat{P}_{n}^{({\rm I},j)}(x), \widehat{P}_{n}^{({\rm II},j)}(x), \widehat{P}_{n}^{({\rm III},j)}(x)$ respectively. Since every $X_j$-OPs are derived from the ordinary COPs via Darboux transformation, it is straightforward to see that,
in the case of when the Darboux parameter is $(\rho,j)$,
\begin{align*}
\begin{array}{ll}
\widehat{P}_{n}^{(\rho ,j)}(x)=\widehat{\phi }_{\gamZero ,n}
&\text{($\rho= {\rm I, II}$), } \\
\widehat{P}_{0}^{({\rm III} ,j)}(x)= \widehat \phi_{{\rm III},j} \propto 1 \text{\,\,and\,\,}
\widehat{P}_{n+j+1}^{({\rm III} ,j)}(x)=\widehat{\phi }_{\gamZero ,n}
& \text{($\rho= {\rm III}$), }
\end{array}
\end{align*}
for $n \in \setZ_{\ge 0}$ .
Therefore, their explicit form is given as follows:
\begin{itemize}
\item (type 1 and type 2 $X_j$-OPs)
\begin{align}\label{t12-xops}
\begin{split}
&\widehat{P}_{n}^{(\rho,j)}(x)=A_{\rho} p_{\rho,j}\left\{ \partial -\frac{\phi_{\rho,j}'}{\phi_{\rho,j}}\right\} P_n(x)
=A_{\rho} p_{\rho,j}\left\{ \partial -\left( \frac{\xi_{\rho }'}{\xi_{\rho }}+\frac{p_{\rho,j}'}{p_{\rho,j}}\right)\right\} P_n(x),\quad (\rho ={\rm I,II})\\
&\deg (\widehat{P}_{n}^{(\rho,j)}(x))=n+j,
\end{split}
\end{align}
\item (type 3 $X_j$-OPs)
\begin{align}\label{t3-xops}
\begin{split}
&\widehat{P}_{n}^{({\rm III},j)}(x)= 
\begin{cases}
\frac{\xi_{\varnothing }}{A_{\varnothing }}\propto 1\quad (n=0)\\
A_{\rm III} \,p_{{\rm III},j}\left\{ \partial -\left( \frac{\xi_{\rm III }'}{\xi_{\rm III}}+\frac{p_{{\rm III},j}'}{p_{{\rm III},j}}\right)\right\} P_{n-j-1}(x)\quad (n\ge j+1)\\
0\quad (\textrm{otherwise})
\end{cases}\\
&\deg (\widehat{P}_{n}^{({\rm III},j)}(x)) = n\quad (n=0,j+1,j+2,\cdots ).
\end{split}
\end{align}
\end{itemize}
Let  $\setN_{\rho,j} = \{n \in \setZ_{\ge 0} \mid \widehat P_n^{(\rho,j)} \ne 0\}$.
Their orthogonality relation is given by
\begin{eqnarray}\label{orthogonality:xops}
\int_{D} \widehat{P}_n^{(\rho,j)}(x)\widehat{P}_{m}^{(\rho,j)}(x) \frac{A(x)\,w(x)}{(A_{\rho }(x)p_{\rho,j}(x))^2}dx=\widehat{h}_{nm}^{(\rho,j)},
\quad \rho ={\rm I,II,III},
\end{eqnarray}
for $n,m \in \setN_{\rho,j}$.

We shall proceed the bispectrality of the $X_j$-OPs, i.e. the recurrence relations. As for the recurrence relations, several results are already discussed. For example, some classical $X_j$-OPs of type 1 and type 2 are shown to satisfy the following $4j+1$-recurrence relation \cite{SasakiTZ2010}:
\begin{align}
(p_{\rho,j}(x))^2\widehat{P}_{n}^{(\rho,j)}(x)= \sum_{l=-2j}^{2j} \alpha_{n,l} \widehat{P}_{n+l}^{(\rho,j)}(x),\quad (\rho ={\rm I,II}).
\end{align}

We have found there exists another recurrence relation, which might be the simplest recurrence formula for XOPs.
\begin{thm}\label{thm1}
For any $\rho \in \{\gamOne, \gamTwo, \gamThree\}$ and $n \in \setN_{\rho,j}$, 
$X_j$-OPs $\{ \widehat{P}_{n}^{(\rho,j)}\}_{n=0}^{\infty }$ satisfy the following $2j+3$-recurrence relation:
\begin{eqnarray}\label{2j+3rec}
I_j(x) \widehat{P}_{n}^{(\rho,j)}(x) = \sum_{l=-j-1}^{j+1} \beta_{n,l} \widehat{P}_{n+l}^{(\rho,j)}(x),
\end{eqnarray}
with $\widehat{P}_{n< 0}^{(\rho,j)}\equiv 0$ and
\[
I_j(x)=\int p_{\rho,j}(x)dx.
\]
\end{thm}
\begin{rem}
In case $j=1$, this gives the same recurrence relation as in \cite{SasakiTZ2010}. In \cite{gomezgm,Odake2013}, $X_j$-OPs are shown to satisfy $5$-term recurrence relation whose coefficients take polynomial forms in $x$, while all the coefficients $\beta_{n,l}$ in \eqref{2j+3rec} are constants. 
\end{rem}
Before the proof, we introduce the following lemma concerning type 3 XOPs.
\begin{lmm}\label{lemma2}
There exist some constants $\alpha, \beta $ s.t. 
\[
I_j(x)= \alpha \widehat{P}_{j+1}^{({\rm III},j)}(x) +\beta \widehat{P}_{0}^{({\rm III},j)}(x).
\]
\end{lmm}
\begin{proof}
Since $\widehat P_{0}^{({\rm III},j)}(x)\propto 1$, it is sufficient to prove
\[
I_j'(x)=p_{{\rm III},j}(x)\propto (\widehat P_{j+1}^{({\rm III},j)}(x))'.
\] 
From the definition \eqref{t3-xops} and $w(x) \xi_{\rm III}(x) =1$, we have
\begin{align*}
(\widehat P_{j+1}^{({\rm III},j)}(x))'
&=\frac{d}{dx}\left( A_{\rm III}(x) p_{{\rm III},j}(x)\left\{ \frac{d}{dx} -\left( \frac{\xi_{\rm III}'(x)}{\xi_{\rm III}(x)}+\frac{p_{{\rm III},j}'(x)}{p_{{\rm III},j}(x)}\right)\right\} P_{0}(x)\right)\\
&\propto \frac{d}{dx}\left(  A_{\rm III}(x) p_{{\rm III},j}(x) \frac{\phi_{{\rm III},j}'(x)}{\phi_{{\rm III},j}(x)} \right) \\
&= \frac{d}{dx}\left( A_{\rm III}(x) w(x)\phi_{{\rm III},j}'(x)\right).
\end{align*}
Recall from lemma \ref{lemma1} $A_{\rm III}(x)=A(x)$ and the equivalent form of the B\"{o}chner type operator \eqref{polyeigen_st}, we see as a consequence
\begin{align*}
(\widehat P_{j+1}^{({\rm III},j)}(x))'\propto \lambda _{{\rm III},j}w(x)\phi _{{\rm III},j}(x)=\lambda _{{\rm III},j}p_{{\rm III},j}(x). 
\end{align*}
This completes the proof.
\end{proof}
\begin{proof}[Proof of Theorem \ref{thm1}]
Due to the completeness of the XOPs, it is possible to write
\[
I_j(x) \widehat{P}_{n}^{(\rho,j)}(x)=\sum_{l=0}^{\infty } \varepsilon_{n,l} \widehat{P}_{l}^{(\rho,j)}(x)
\]
uniquely. From the above expansion, we only have to prove
\[
\varepsilon_{n,l}=0,\quad (|l-n|>j+1).
\]
Using the orthogonality relation \eqref{orthogonality:xops}, one can find 
\[
\varepsilon_{n,l}=(\widehat{h}_l^{(\rho,j)})^{-1} \int_{D} I_j(x)\widehat{P}_{n}^{(\rho,j)}(x)\widehat{P}_{l}^{(\rho,j)}(x) \frac{A(x)w(x)}{(A_{\rho }(x)p_{\rho,j}(x))^2}dx.
\]
In the followings, we shall examine the quantity $\varepsilon_l$ for each case.
\begin{itemize}
\item ($\rho ={\rm I,II}$) From the definition \eqref{t12-xops}, we have
\begin{align*}
\widehat{h}_l^{(\rho,j)}\varepsilon_{n,l}
&=\int_{D} I_j\left( P_n'-\frac{\phi_{\rho,j}'}{\phi_{\rho,j}}P_n\right) \left( P_l'-\frac{\phi_{\rho,j}'}{\phi_{\rho,j}}P_l\right) Awdx\\
&=\int_{D} I_jP_n'P_l'Aw-\int_{D} \partial (P_lP_n) \frac{\phi_{\rho,j}'}{\phi_{\rho,j} }I_jAw dx+\int_{D} P_lP_n \left( \frac{\phi_{\rho,j}'}{\phi_{\rho,j} }\right)^2 I_jAwdx.
\end{align*}
From the property of COPs \eqref{derivative-ops}, recalling $I_j$ is the polynomial in $x$ of $j+1$ degree, it is easy to see
\[
\int_{D} I_jP_n'P_l'Awdx=0\quad (|n-l|>j+1).
\] 
Then for $|l-n|>j+1$, we have performing the partial integral
\begin{align*}
\widehat{h}_l^{(\rho,j)}\varepsilon_{n,l}
&=-\int_{D} \partial (P_lP_n) \frac{\phi_{\rho,j}'}{\phi_{\rho,j} }I_jAw dx+\int_{D} P_lP_n \left( \frac{\phi_{\rho,j}'}{\phi_{\rho,j} }\right)^2 I_jAwdx\\
&=\int_{D} P_lP_n \partial \left( \frac{\phi_{\rho,j}'}{\phi_{\rho,j} }I_jAw \right)dx+ \int_{D} P_lP_n \left( \frac{\phi_{\rho,j}'}{\phi_{\rho,j} }\right)^2 I_jAwdx\\
&=\int_{D} P_lP_n\left( \frac{\phi_{\rho,j}''}{\phi_{\rho,j} }I_jAw +  \frac{\phi_{\rho,j}'}{\phi_{\rho,j} }p_{\rho,j}Aw+ \frac{\phi_{\rho,j}'}{\phi_{\rho,j} }I_jBw\right)dx\\
&=\int_{D} P_lP_nI_j \frac{A\phi_{\rho,j}''+B\phi_{\rho,j}'}{\phi_{\rho,j}}wdx+\int_{D} P_lP_n\left( \frac{\xi_{\rho }'}{\xi_{\rho }}+\frac{p_{\rho,j}'}{p_{\rho,j}}\right) Ap_{\rho,j} wdx.
\end{align*}
From \eqref{bochner:seed}, we have for $|l-n|>j$,
\begin{align*}
\widehat{h}_l^{(\rho,j)} \varepsilon_{n,l}
&=\int_{D} P_lP_n( \lambda_{\rho,j}I_j+Ap_{\rho,j}'+\eta _{\rho }Ap_{\rho,j})wdx.
\end{align*}
From the lemma \ref{lemma1}, one can verify that $A\eta_{\rho } $ is a polynomial in $x$ of degree up to 1.
Then, $ \deg( \lambda_{\rho,j}I_j+Ap_{\rho,j}'+\eta _{\rho }Ap_{\rho,j})=j+1$, which amounts to
\[
\varepsilon_{n,l}=0,\quad |n-l|>j+1.
\]
\item ($\rho ={\rm III}$) When $n=0$, from the lemma \ref{lemma2}, we directly have
\begin{eqnarray}\label{2j+3-n=0}
I_j(x) \widehat{P}_{0}^{({\rm III},j)}(x)= \alpha \widehat{P}_{j+1}^{({\rm III},j)}(x) +\beta \widehat{P}_{0}^{({\rm III},j)}(x),\quad (\exists \alpha,\beta \in \mathbb{C}).
\end{eqnarray}
This is nothing but the $2j+3$-recurrence relation. For $n\ge j+1$, in the fashion similar to the case $\rho ={\rm I,II}$, one can find
\[
\varepsilon_{n,l}=0,\quad |n-l|>j+1 \quad \textrm{for }\quad l\ge j+1.
\]
Furthermore, from \eqref{2j+3-n=0}, it is straightforward to see from the orthogonality relation \eqref{orthogonality:xops}
\begin{align*}
\varepsilon_{n,0}
&=(\widehat{h}_0^{(\rho,j)})^{-1} \int_{D} I_j\widehat{P}_{n}^{(\rho,j)}\widehat{P}_{0}^{(\rho,j)} \frac{Aw}{(A_{\rho }p_{\rho,j})^2}dx\\
&=(\widehat{h}_0^{(\rho,j)})^{-1} \int_{D} \widehat{P}_{n}^{(\rho,j)}(\alpha \widehat{P}_{j+1}^{(\rho,j)}+\beta \widehat{P}_{0}^{(\rho,j)}) \frac{Aw}{(A_{\rho }p_{\rho,j})^2}dx\\
&=0\quad (n> j+1).
\end{align*}
Therefore, for $n\ge 1$, we have as a consequence
\[
\varepsilon_{n,l}=0,\quad |n-l|>j+1.
\]
\end{itemize}
Combining these results, we have the $2j+3$-recurrence relations for every type $X_j$-OPs. This completes the proof.
\end{proof}
It is worth noting that from the proof we can evaluate all the coefficients of the recurrence relation in terms of the orthogonality constant with respect to the ordinary COPs. We give the collection of the explicit expression of the recurrence relation for some classical XOPs in the following section. It should also be remarked if we take $j=0$, which corresponds to the ordinary OPs, the recurrence relation \eqref{2j+3rec} becomes the three term recurrence relation as ordinary OPs satisfy. In that sense, the recurrence relation \eqref{2j+3rec} is a natural generalization of the three term recurrence relation for OPs.

\section{Recurrence formulas for the classical XOPs}


In this section, as is mentioned in the previous section, we give the explicit form of the recurrence relation for the classical $X_j$-OPs including Hermite, Laguerre and Jacobi polynomials. In order to avoid the complexity of the recurrence relation, we especially fix all polynomials (including OPs, XOPs) to monic ones.

The ordinary Hermite, Laguerre and Jacobi polynomials in monic form are defined by the Rodrigues' formula\cite{Koekoek}: 
\begin{eqnarray*}
 &&H_n(x) = \dfrac{(-1)^n  }{2^n} e^{x^2}\left(\dfrac{d}{dx}\right)^n e^{-x^2},\\
 &&L_n^{(a)}(x) = (-1)^n\dfrac{e^{x}}{x^{a}} 
\left(\dfrac{d}{dx}\right)^n \left[\dfrac{x^{n+a}}{e^{x}}\right],\\
 && J_n^{(a,b)}(x)=  \dfrac{(-1)^n}{(n+a+b+1)_n}
(1-x)^{-a}(1+x)^{-b}
\left(\dfrac{d}{dx}\right)^n [(1-x)^{n+a}(1+x)^{n+b}],
\end{eqnarray*}
respectively, where $(a)_n=a(a+1)\cdots (a+n-1)$ is the standard Pochhammer symbol. In the followings, we give the corresponding recurrence formula \eqref{2j+3rec} in each case.

\paragraph{$X_j$-Hermite polynomials}
\begin{enumerate}[1.]
\item There are no type-1 exceptional Hermite polynomials.
\item There are no type-2 exceptional Hermite polynomials.
\item Type-3 $X_j$-Hermite polynomials are introduced as follows ($j=2,4,6,\cdots $):
\[
\widehat{H}_n^{({\rm III},j)}(x)=
\left\{\begin{array}{cl}
1\quad &(n=0)\\ 
-2^{-1} i^{-j}
\left(H_j(ix)(\partial -2x)H_{n-j-1}(x)-\partial (H_j(ix))H_{n-j-1}(x)\right)\quad & (n \ge j+1)\\
0\quad &(\textrm{otherwise})\\ 
       \end{array}\right.
\]
with $i=\sqrt{-1}$ and the polynomial part of the seed function and decoupling factor chosen
\[
\phi_{\gamThree,j} (x)=e^{x^2}H_j(ix),\quad \pi (x)=H_j(ix).
\]
It is easy to see 
\[
\int p_{{\rm III},j}(x)dx=\frac{-i}{j+1}H_{j+1}(ix)+C,
\]
which will give us $2j+3$-recurrence relations:
\begin{itemize}
\item $j=2$: $7$-term recurrence relation $(n\ne 1,2)$

\begin{align*}
iH_3(ix)\,\widehat H_{n}^{\rm (III,2)}(x)=&
  \widehat H_{n+3}^{\rm (III,2)}(x)+\frac{3}{2}\, n\, \widehat H_{n+1}^{\rm (III,2)}(x)
+\frac{3}{4}\, n \,(n-3) \widehat H_{n-1}^{\rm (III,2)}(x)
\\
&+\frac{1}{8}\, n \,(n-4) (n-5) \widehat H_{n-3}^{\rm (III,2)}(x)
\end{align*}
with $iH_3(ix)=x^3+3x/2 $.
\end{itemize}
\end{enumerate}

\paragraph{$X_j$-Laguerre polynomials}

\begin{enumerate}[1.]
\item Type-1 $X_j$-Laguerre polynomials are introduced as follows:
\[
\widehat{L}_{n}^{({\rm I},j)}(x;a)= 
(-1)^{j+1}
\left(L_j^{(a)}(-x)(\partial - 1)L_n^{(a)}(x) - \partial (L_j^{(a)}(-x))L_n^{(a)}(x)\right),
\]
where the polynomial part of the seed function and decoupling factor are chosen
\[
\phi_{\gamOne,j} (x)=e^x L_j^{(a)}(-x),\quad \pi (x)=L_j^{(a)}(-x),
\]
respectively. It is easy to verify
\[
\int p_{{\rm I}, j}(x)dx=-\frac{1}{j+1}L_{j+1}^{(a)}(-x)+C,
\]
which will give us $2j+3$-recurrence relations:
\begin{itemize}
\item $j=1$: $5$-term recurrence relation

\begin{align*}
&
L_2^{(a-1)}(-x) \widehat L_{n}^{\rm (I,1)}(x;a)
= \widehat L_{n+2}^{\rm (I,1)}(x;a)
+4 (n+a+2) \widehat L_{n+1}^{\rm (I,1)}(x;a)
\\&\quad
+2 (n+a+2) (3 n+2 a+2) \widehat L_{n}^{\rm (I,1)}(x;a)
+4 (n+a+2) (n+a) n \widehat L_{n-1}^{\rm (I,1)}(x;a)
\\&\quad
+  (n+a+2) (n+a-1) n (n-1) \widehat L_{n-2}^{\rm (I,1)}(x;a)
\end{align*}
with $L_2^{(a-1)}(-x) =x^2+ (2a+2) x+a(a+1)$,
\item $j=2$: $7$-term recurrence relation 
\begin{align*}
&-L_3^{(a-1)}(-x) \widehat L_{n}^{\rm (I,2)}(x;a)\\
&=\widehat L_{n+3}^{\rm (I,2)}(x;a)
+6 (n+a+3) \widehat L_{n+2}^{\rm (I,2)}(x;a)
+3 (n+a+3) (5 n+4 a+8) \widehat L_{n+1}^{\rm (I,2)}(x;a)
\\&\quad
+4 (n+a+3) (n+a+1)   (5 n+2 a+4)  \widehat L_{n}^{\rm (I,2)}(x;a)
+3 (n+a+3) (n+a  )   (5 n+4 a+3) n \widehat L_{n-1}^{\rm (I,2)}(x;a)
\\&\quad
+6 (n+a+3) (n+a-1)_2 (n-1)_2 \widehat L_{n-2}^{\rm (I,2)}(x;a)
+  (n+a+3) (n+a-2)_2 (n-2)_3 \widehat L_{n-3}^{\rm (I,2)}(x;a)
\end{align*}
with $-L_3^{(a-1)}(-x)=x^3+3(a+2) x^2+3 (a+1) (a+2) x+a (a+1) (a+2) $.
\end{itemize}
\item Type-2 $X_j$-Laguerre polynomials are introduced as follows:
\[
\widehat{L}_{n}^{({\rm II},j)}(x;a)=(n+a-j)^{-1}\left(L_j^{(-a)}(x)(x\partial +a)L_n^{(a)}(x)-x\partial (L_j^{(-a)}(x))L_n^{(a)}(x)\right),
\]
where the polynomial part of the seed function and decoupling factor are chosen
\[
\phi_{\gamTwo,j} (x)=x^{-a} L_j^{(-a)}(x),\quad \pi (x)=xL_j^{(-a)}(x),
\]
respectively. It is easy to verify
\[
\int p_{{\rm II}, j}(x)dx=\frac{1}{j+1}L_{j+1}^{(-a-1)}(x)+C,
\]
which will give us $2j+3$-recurrence relations:
\begin{itemize}
\item $j=1$: $5$-term recurrence relation

\begin{align*}
&
L_2^{(-a-1)}(x)\,\widehat L_{n}^{\rm (II,1)}(x;a)\\
&= \widehat L_{n+2}^{\rm (II,1)}(x;a)
+4 (n+a  ) \widehat L_{n+1}^{\rm (II,1)}(x;a)
+2 (n+a-1) (3 n+2 a+1) \widehat L_{n}^{\rm (II,1)}(x;a)
\\&\quad
+4 (n+a-2) (n+a)\, n\, \widehat L_{n-1}^{\rm (II,1)}(x;a)
+(n+a-3) (n+a) \,n \,(n-1) \widehat L_{n-2}^{\rm (II,1)}(x;a)
\end{align*}
where $L_2^{(-a-1)}(x)=x^2+(2 a-2) x+a(a-1)$,
\item $j=2$: $7$-term recurrence relation 

\begin{align*}
&
L_3^{(-a-1)}(x)\,\widehat L_{n}^{\rm (II,2)}(x;a)\\
&=
\widehat L_{n+3}^{\rm (II,2)}(x;a)
+6 (n+a  ) \widehat L_{n+2}^{\rm (II,2)}(x;a)
+3 (n+a-1) (5 n+4 a+2) \widehat L_{n+1}^{\rm (II,2)}(x;a)
\\&\quad
+4 (n+a-2) (n+a) (5 n+2 a+1) \widehat L_{n}^{\rm (II,2)}(x;a)
+3 (n+a-3) (n+a) (5 n+4 a-3) n \widehat L_{n-1}^{\rm (II,2)}(x;a)
\\&\quad
+6 (n+a-4) (n+a-1)_2 (n-1)_2 \widehat L_{n-2}^{\rm (II,2)}(x;a)
+  (n+a-5) (n+a-1)_2 (n-2)_3 \widehat L_{n-3}^{\rm (II,2)}(x;a)
\end{align*}
where $L_3^{(-a-1)}(x)=x^3+3(a-2) x^2+3 (a-1)(a-2)  x+ a (a-1)(a-2)$.
\end{itemize}
\item Type-3 $X_j$-Laguerre polynomials are introduced as follows:
\[
\widehat{L}_{n}^{({\rm III},j)}(x;a)
=
\left\{\begin{array}{cl}
1 \quad &(n=0)\\
(-1)^{j+1}\left(L_j^{(-a)}(-x)(x\partial +a-x)L_{n-j-1}^{(a)}(x)-x\partial (L_j^{(-a)}(-x))L_{n-j-1}^{(a)}(x) \right) \quad &(n \ge j+1)\\
0 \quad &(\textrm{otherwise})\\
       \end{array}\right.,
\]
where the polynomial part of the seed function and decoupling factor are chosen
\[
\phi_{\gamThree,j} (x)=x^{-a}e^x L_j^{(-a)}(-x),\quad \pi (x)=xL_j^{(-a)}(-x),
\]
respectively. It is easy to verify
\[
\int p_{{\rm III}, j}(x)dx=-\frac{1}{j+1}L_{j+1}^{(-a-1)}(-x)+C,
\]
which will give us $2j+3$-recurrence relations:
\begin{itemize}
\item $j=1$: $5$-term recurrence relation ($n\ne 1$)
\begin{align*}
&
L_2^{(-a-1)}(-x)\,\widehat L_{n}^{\rm (III,1)}(x;a)
\\&
=\widehat L_{n+2}^{\rm (III,1)}(x;a)
+4 n \widehat L_{n+1}^{\rm (III,1)}(x;a)
+2 n (3 n+a-4) \widehat L_{n}^{\rm (III,1)}(x;a)
\\&\quad
+4 n (n-2) (n+a-2) \widehat L_{n-1}^{\rm (III,1)}(x;a)
+  n (n-3) (n+a-2) (n+a-3) \widehat L_{n-2}^{\rm (III,1)}(x;a)
\end{align*}
where $L_2^{(-a-1)}(-x)=x^2+ (-2a+2) x+a(a-1)$,
\item $j=2$: $7$-term recurrence relation ($n\ne 1,2$)
\begin{align*}
&-L_3^{(-a-1)}(-x)\,\widehat L_{n}^{(\rm III,2)}(x;a)\\
&=
\widehat L_{n+3}^{(\rm III,2)}(x;a)
+6n \widehat L_{n+2}^{(\rm III,2)}(x;a)
+3n (5n+a-7) \widehat L_{n+1}^{(\rm III,2)}(x;a)
\\&\quad
+4n (n-2) (5n+3a-11) \widehat L_{n}^{(\rm III,2)}(x;a)
+3n (n-3) (n+a-3) (5n+a-12)\widehat L_{n-1}^{(\rm III,2)}(x;a)
\\&\quad
+6n (n-4)_2 (n+a-4)_2 \widehat L_{n-2}^{(\rm III,2)}(x;a)
+ n (n-5)_2 (n+a-5)_3 \widehat L_{n-3}^{(\rm III,2)}(x;a)
\end{align*}
where $-L_3^{(-a-1)}(-x)=x^3-3(a-2) x^2+3 (a-1) (a-2) x-a(a-1)(a-2)$.
\end{itemize}
\end{enumerate}

\paragraph{$X_j$-Jacobi polynomials}
As for $X_j$-Jacobi polynomials, 
the corresponding recurrence relations take simple form if $a=b$. 
We thus restrict the case to $a=b$.
\begin{enumerate}
\item Type-1 $X_j$-Jacobi polynomials are introduced as follows:
\[
\widehat{J}_{n}^{{\rm (I,1)}}(x;a,b)
=
\frac{J_j^{(a,-b)}(x)((1+x)\partial +b)J_n^{(a,b)}(x)-(1+x)\partial(J_j^{(a,-b)}(x))J_n^{(a,b)}(x)}{n+b-j}
\]
where the polynomial part of the seed function and decoupling factor are chosen
\[
\phi_{\gamOne,j} (x)=(1+x)^{-b} J_j^{(a,-b)}(x),\quad \pi (x)=(1+x)J_j^{(a,-b)}(x),
\]
respectively. It is easy to verify
\[
\int p_{{\rm I}, j}(x)dx=\frac{1}{j+1}J_{j+1}^{(a-1,-b-1)}(x)+C,
\]
which will give us $2j+3$-recurrence relations:
\begin{itemize}
\item $j=1~(a=b)$: 5-term recurrence relation
\begin{align*}
 &
 J_{2}^{(a-1,-a-1)}(x)\, \widehat J_{n}^{({\rm I},1)}(x;a,a)
=\widehat J_{n+2}^{({\rm I},1)}(x;a,a)
+2a\widehat J_{n+1}^{({\rm I},1)}(x;a,a)
\\&\quad
+\frac{2(4a^2-1)(n+a-1)(n+a+2)}{(2n+2a-1)(2n+2a+3)}
   \widehat J_{n}^{({\rm I},1)}(x;a,a)
+\frac{8a((n+a)^2-4)\,n\,(n+2a)}{(2n+2a-2)_2(2n+2a+1)_2}
   \widehat J_{n-1}^{({\rm I},1)}(x;a,a)
\\&\quad
+\frac{4(n+a-3)(n+a+2)(n-1)_2(n+2a-1)_2}{(2n+2a-3)_3(2n+2a-1)_3}
   \widehat J_{n-2}^{({\rm I},1)}(x;a,a)
\end{align*}
where $J_{2}^{(a-1,-a-1)}(x)=x^2+2ax+2a^2-1$,
\item $j=2~(a=b)$: 7-term recurrence relation
\begin{align*}
&
 J_3^{(a-1,-a-1)}(x) \widehat J_{n}^{({\rm I},2)}(x;a,a)
=\widehat J_{n+3}^{({\rm I},2)}(x;a,a)
+\dfrac{3a}{2}\widehat J_{n+2}^{({\rm I},2)}(x;a,a)
\\&\quad
+
{\frac{\left(4\,a^2-1\right)\left(n+a+3\right)\left(n+a-1\right) }
{\left(2\,n+2\,a-1\right)\left(2\,n+2\,a+5\right)}}
\widehat J_{n+1}^{({\rm I},2)}(x;a,a)
\\&\quad
+
{\frac {a\left(4\,a^2-1\right)\left(n+a+3\right)\left(n+a-2 \right)}
        {3\left(2\,n+2\,a-1 \right)\left(2\,n+2\,a+3 \right)}}
\widehat J_{n}^{({\rm I},2)}(x;a,a)
\\&\quad
+
{\frac {\left(4\,a^2-1\right)\left(n+a+3\right)\left(n+a-3 \right)}
        {\left(2\,n+2\,a-1 \right)\left(2\,n+2\,a+2 \right)}}
{\frac { n \left(n+2a\right)}
        {\left(2\,n+2\,a-3 \right)\left(2\,n+2\,a+3 \right)}}
\widehat J_{n-1}^{({\rm I},2)}(x;a,a)
\\&\quad
+
{\frac {6 a \left(n+a+3\right)\left(n+a-4 \right)}
        {\left(2\,n+2\,a-1 \right)^2}}
{\frac { (n-1)_2 \left(n+2a-1\right)_2}
        {\left(2\,n+2\,a-4 \right)_2\left(2\,n+2\,a+1 \right)_2}}
\widehat J_{n-2}^{({\rm I},2)}(x;a,a)
\\&\quad
+
{\frac {4 \left(n+a+3\right)\left(n+a-5 \right)}
        {\left(2\,n+2\,a-1 \right)\left(2\,n+2\,a-3 \right)}}
{\frac { (n-2)_3 \left(n+2a-2\right)_3}
        {\left(2\,n+2\,a-5 \right)_3\left(2\,n+2\,a-1 \right)_3}}
\widehat J_{n-3}^{({\rm I},2)}(x;a,a)
\end{align*}
where $J_3^{(a-1,-a-1)}(x)=
x^3+3ax^2/2+(a-1)(a+1)x+a(2a^2-5)/6$,
\end{itemize}

\item Type-2 $X_j$-Jacobi polynomials are introduced as follows:
\[
\widehat{J}_{n}^{({\rm II},j)}(x;a,b)
=\frac{J_j^{(-a,b)}(x)((1-x)\partial -a)J_n^{(a,b)}(x)-(1-x)\partial(J_j^{(-a,b)}(x))J_n^{(a,b)}(x)}{-n-a+j}
\]
where the polynomial part of the seed function and decoupling factor are chosen
\[
\phi_{\gamTwo,j} (x)=(1-x)^{-a} J_j^{(-a,b)}(x),\quad \pi (x)=(1-x)J_j^{(-a,b)}(x),
\]
respectively.
In this case the recurrence relations  can be immediately derived by observing
\begin{align*}
 \widehat{J}_{n}^{({\rm II},j)}(x;a,a) = (-1)^{n+j}\widehat{J}_{n}^{({\rm I},j)}(-x;a,a) .
\end{align*}
\item Type-3 $X_j$-Jacobi polynomials are introduced as follows:
\[
\widehat{J}_{n}^{({\rm III},j)}(x;a,b)
=
\left\{\begin{array}{cl}
1 &(n=0)\\
\bigl(J_j^{(-a,-b)}(x)\,((x^2-1)\partial +(a+b)x+a-b)J_{n-j-1}^{(a,b)}(x)\\
+(1-x^2)\partial(J_j^{(-a,-b)}(x))J_{n-j-1}^{(a,b)}(x)\bigr)/
(n+a+b-2j-1)
& (n \ge j+1)
\\
0 & (\textrm{otherwise})
\end{array}\right.
\]
where the polynomial part of the seed function and decoupling factor are chosen
\[
\phi_{\gamThree,j} (x)=(1-x)^{-a}(1+x)^{-b} J_j^{(-a,-b)}(x),\quad \pi (x)=(1-x)^2J_j^{(-a,-b)}(x),
\]
respectively. It is easy to verify
\[
\int p_{{\rm III}, j}(x)dx=\frac{1}{j+1}J_{j+1}^{(-a-1,-b-1)}(x)+C,
\]
which will give us $2j+3$-recurrence relations:
\begin{itemize}
\item $j=1~(a=b)$: 5-term recurrence relation ($n\ne 1$)
\begin{align*}
&J_{2}^{(-a-1,-a-1)}(x)\,\widehat{J}_{n}^{({\rm III},1)}(x;a,a)
=\widehat{J}_{n+2}^{({\rm III},1)}(x;a,a)
\\&\quad
+\dfrac{2(2a+1)n(n+2a-3)}{(2a-1)(2n+2a-5)(2n+2a-1)}\widehat{J}_{n}^{({\rm III},1)}(x;a,a)
\\&\quad
+\dfrac{n(n-3)(n+2a-5)(n+2a-2)}{8(2n+2a-5)(n+a-7/2)_3}
\widehat{J}_{n-2}^{({\rm III},1)}(x;a,a)
\end{align*}
where
$J_{2}^{(-a-1,-a-1)}(x)=x^2+1/(2a-1)$,
\item $j=2~(a=b)$: 7-term recurrence relation ($n\ne 1,2$)
\begin{align*}
&J_{3}^{(-a-1,-a-1)}(x)\,\widehat{J}_{n}^{({\rm III},2)}(x;a,a)
=
\widehat{J}_{n+3}^{({\rm III},2)}(x;a,a)
\\&\quad
+\dfrac{3(2a+1)n(n+2a-4)}{(2a-3)(2n+2a-1)(2n+2a-7)}
\widehat{J}_{n+1}^{({\rm III},2)}(x;a,a)
\\&\quad
+\dfrac{3(2a+1)n(n-3)(n+2a-3)(n+2a-6)}{2^4(2a-3)(n+a-9/2)_4}
\widehat{J}_{n-1}^{({\rm III},2)}(x;a,a)
\\&\quad
+\dfrac{n(n-4)(n-5)(n+2a-8)(n+2a-4)_2}{2^6 (n+a-9/2)_2(n+a-11/2)_4}
\widehat{J}_{n-3}^{({\rm III},2)}(x;a,a)
\end{align*}
where
$J_{3}^{(-a-1,-a-1)}(x)= x^3+3x/(2a-3)$.
\end{itemize}
\end{enumerate}

\section{Concluding Remarks}
In this paper, we have given a brief review of XOPs especially focusing on possible quasi-polynomial solutions to the B\"{o}chner type operator and the corresponding Darboux transformations, from which $X_j$-OPs of all kinds are obtained. Observing the structure of the Darboux transformation carefully, we have shown a new recurrence formula for $X_j$-OPs. We have also given the explicit form of the recurrence relation for several classical $X_j$-OPs.

The recurrence formula holds for every type $X_j$-OPs and all the coefficients except spectra of the recurrence formula are constants. This means that $X_j$-OPs diagonalize some band-matrices which can be regarded as the generalized Jacobi matrix, which we believe produces the exactly solvable physical model. Recently, multi-indexed OPs are derived from the multi-step Darboux-transformations and the recurrence formula with constant coefficients for multi-indexed OPs are expected to be obtained by our strategy\cite{OS2011multi,gomez3}. In \cite{Duran}, the higher-order recurrence relation is given from the ``dual'' aspects of OPs with higher order difference equation, which is definitely related to our recurrence formula. We shall plan to return to these in the future publications. 
\section*{Acknowledgements}
The authors thank to A.~Duran, R.~Milson,S.~Odake, R.~Sasaki, L.~Vinet and A.~Zhedanov
for fruitful discussions and helpful advice. 
They also give special thank to D.~G\'omez-Ullate, F.~Marcell\'an and M.~Rodr\'iguez for organizing 
the workshop on ``XOPs and exact solutions in mathematical physics'' at Segovia, Spain.
The research of ST was supported in part by JSPS KAKENHI Grant Numbers 25400110.


\end{document}